\documentclass[reqno,a4paper,11pt]{amsart}
\usepackage{amssymb}
\usepackage[dvips, lmargin=3cm, rmargin=3cm, tmargin=4cm, bmargin=4cm]{geometry}
\usepackage[colorlinks=true, citecolor=red, linkcolor=blue, urlcolor=blue]{hyperref}
\usepackage{cite}
\usepackage{xfrac}
\usepackage{graphicx}
\usepackage{hyperref}
\usepackage{amsmath}
\usepackage{mathrsfs}
\usepackage[T1]{fontenc}
\usepackage{enumitem}
\usepackage{dsfont}
\everymath{\displaystyle}
\newtheorem{theorem}{Theorem}[section]
\newtheorem{definition}{Definition}[section]
\newtheorem{remark}{Remark}[section]
\newtheorem{lemma}{Lemma}[section]
\newtheorem{proposition}{Proposition}[section]

\numberwithin{equation}{section}

\begin{document}
\title[Embedding Results and Nonlocal Problem in Unbounded Domains]{Embedding Results and Fractional $p(x,.)$-Laplacian Problem in Unbounded Domains}
\author[A. BARBARA, A. BOUSMAHA, and M. SHIMI]
{A. BARBARA$^1$, A. BOUSMAHA$^1$, and  M. SHIMI$^2$, }
\noindent  \address{Abdelkrim  BARBARA, Ahmed BOUSMAHA\newline Faculty of Sciences Dhar el Mahraz, Laboratory of Modeling, Applied Mathematics, and Intelligent Systems, Sidi Mohamed Ben Abdellah University, Fez, Morocco}
 \email{$^{1}$abdelkrim.barbara@usmba.ac.ma} 
  \email{$^{2}$ahmed.bousmaha@usmba.ac.ma}
\noindent  \address{Mohammed Shimi\newline
Department of Mathematics and Computer Science, Laboratory of Mathematics and Applications to Engineering Sciences, ENS of Fez, Sidi Mohamed Ben Abdellah University, Fez, Morocco.}
\email{$^3$mohammed.shimi2@usmba.ac.ma }
\subjclass[2010]{35R11, 47G20, 35S15, 35A15.}
\keywords{Fractional Sobolev spaces with variable exponent, Unbounded domains, Embedding results, Nonlocal problems, Mountain Pass Theorem}
\date{Month, Day, Year}
\begin{abstract}
In this paper, we prove a new continuous embedding theorem for fractional Sobolev spaces with variable exponents into variable exponent Lebesgue spaces on unbounded domains.
 As an application, we study a class of nonlocal elliptic problems driven by the fractional $p(x, \cdot)$-Laplacian operator. Using variational methods combined with the established embedding result, we prove the existence of nontrivial weak solutions under suitable growth and regularity conditions on the nonlinearity.\\  
A significant analytical challenge addressed in this work arises  from   both the nonlocal behavior of the fractional $p(x, \cdot)$-Laplacian and the lack of compactness induced by the unboundedness of the domain.
\end{abstract}
\maketitle
\tableofcontents
\section{Introduction and statement of the main results}
Lebesgue and Sobolev spaces ($L^p, W^{m,p} \text{ with } m\in \mathbb{N},$ \cite{Adams}) were and are the backbone (in the widest sense of the word) of the theory of function spaces and PDE's from the very beginning up to our time. However, classical Lebesgue and Sobolev spaces face limitations when dealing with nonlinear problems in applied sciences.  \\

To overcome these challenges, variable exponent Lebesgue-Sobolev spaces $(L^{p(x)}, W^{m,p(x)})$ were introduced. These spaces are not only important in theoretical studies but also useful for modeling problems with variable exponent growth, which arise in various scientific applications (see for instance \cite{Diening3,KovRak,Ruz}).  \\

Another significant extension is provided by fractional Sobolev spaces$W^{s,p}$, with $s \in (0,1)$ \cite{DiNezza}, play a key role in functional and harmonic analysis and have been widely applied in real-world problems.  
As a continuation of these advancements, many scholars combined the variable exponent setting with fractional case for introducing the fractional Sobolev spaces with variable exponent $W^{s,p(x,y)}$ which opening new perspectives in functional and harmonic analysis and its applications.\\

\noindent The qualitative functional properties, as well as some continuous and compact embedding theorems of these spaces into variable exponent Lebesgue spaces and H\"{o}lder spaces, have been established. 
Precisely, Kaufmann et al.  \cite{Kaufmann} introduced the new fractional Sobolev spaces with variable exponents  $W^{s, q(x), p(x,y)}(\Omega)$ with the restriction $p(x, x)<q(x)$, that takes into account a fractional variable exponent operator $(-\Delta)_{p(x,.)}^{s}$. Moreover, they established a continuous and compact embedding theorem of these spaces into variable exponent Lebesgue spaces. As an application, they studied the existence and uniqueness of a solution for a nonlocal problem involving the fractional $p(x,.)$-Laplacian operator. In this context, the authors  in \cite{Bahrouni1}  investigated fundamental properties of these function spaces and the related nonlocal operator  $(-\Delta)_{p(x,.)}^{s}$.  They also show the existence of at least one solution for equations involving the fractional $p(x,.)$-Laplacian. Subsequently,  Ho and Kim \cite{Ho} proved some fundamental embedding results for the fractional Sobolev space with variable exponent in the case $p(x, x)\leq q(x)$ along with applications such as a priori bounds and multiplicity of solutions of the fractional $p(x,\cdot)$-Laplacian problems. The authors in \cite{Rossi} considered the case $q(x)=p(x,x)$ and established a trace theorem in fractional spaces with variable exponents.\\

 Accordingly, In \cite{ABS2} , the authors considered  the case $p(x,x)=q(x)$  which represents the true generalization of the fractional Sobolev space  $W^{s,p}$ in the constant case. Afterward, they provided a new version of continuous and compact embedding of theses spaces into generalized Lebesgue spaces. In \cite{ABS9} they further established  an embedding results into H\"{o}lder spaces $\mathscr{C}^{0,r}_b$. \\
 
 Since The space  $W^{s, p(x,y)}(\Omega)$ is not suitable for studying the fractional $p(x,.)$-Laplacian problems with Dirichlet boundary data $u=0$ in $\mathbb{R}^{N} \setminus \Omega$ via variational methods and hence, \cite{ABS4, ABS8}  introduced a more appropriate space, denoted by $W^{s,p(x,y)}(Q)$. For this new  space $W^{s,p(x,y)}(Q)$, they established their qualitative properties and compared it with the space $W^{s,p(x,y)}(\Omega)$ Furthermore, they extended all the continuous and compact embedding results previously established for $W^{s,p(x,y)}(\Omega)$ into the spaces $W^{s,p(x,y)}(Q)$ and $W_K^{s,p(x,y)}(Q)$.\\

To the best of our knowledge, all the aforementioned results have been considered in the case of a \textit{bounded} open domain $\Omega$ of $\mathbb{R}^N$. This naturally raises the question : \textit{what happens if we replace the bounded domain with an unbounded one}?  \\ As in the classical setting, studying nonlocal problems in unbounded domains presents several challenges compared to bounded domains. One of the main difficulties is the lack of compactness, as many fundamental functional analysis tools, such as the Rellich-Kondrachov compactness theorem, no longer hold. As a result, weak convergence results become more difficult to establish in this framework.\\

Hence, our main aim in this paper is to start addressing these new challenges by considering an \textit{unbounded domain $\Omega$ with cone property}. Our goal is to establish an embedding theorem and discuss the existence of nontrivial solutions for a fractional  $p(x,\cdot)$-Laplacian problem in this setting.\\

Now,	Given  an unbouded domain $\Omega$  with cone condition (see Definition \ref{00}). We start by fixing $s \in (0,1)$ and let $p:\overline{\Omega}\times\overline{\Omega}\longrightarrow (1,+\infty)$ be an uniformly continuous function such that 
	\begin{equation}\label{n}
		1<p^{-}=\inf_{(x,y)\in \overline{\Omega}\times\overline{\Omega}} p(x,y)\leq p(x,y)\leq p^{+}=\sup_{(x,y)\in \overline{\Omega}\times\overline{\Omega} }p(x,y)<+\infty
	\end{equation}
	and
	\begin{equation}\label{l} 
		p\text{ is symmetric, that is, }p(x,y)=p(y,x) \quad \text{for all }(x,y)\in \overline{\Omega}\times\overline{\Omega}.
	\end{equation}
	We set 
	$$
	\overline{p}(x)=p(x,x) \quad \text{for all }x\in \overline{\Omega}.
	$$
In order to state our main embedding result, we define the fractional Sobolev space with variable exponent via the Gagliardo approach  as follows :
		\begin{align*}
		W &=W^{s,p(x,y)}(\Omega)\\
		&=\left\{ u \in L^{\overline{p}(x)}(\Omega) \text{ :}
		\,\int_{\Omega\times\Omega} 
		\frac{|u(x)-u(y)|^{p(x,y)}}{\lambda^{p(x,y)}|x-y|^{N+sp(x,y)}}\,dxdy<\infty,\,\, \text{for some }\lambda>0
		\right\},
		\end{align*}
	where $L^{\overline{p}(x)}(\Omega)$ is the Lebegue space with variable exponent (see section $\ref{21}$).\\ 
	The space $W^{s,p(x,y)}(\Omega)$ is a Banach space (see  \cite{ABS2,Rossi,Kaufmann}) if it is equipped with the norm
	$$
	\|u\|_{W}=\|u\|_{L^{\overline{p}(x)}(\Omega)}+[u]_{s,p(x,y)},
	$$
	where $[.]_{s,p(x,y)}$ is a Gagliardo seminorm with variable exponent, which is defined by
	$$
	[u]_{s,p(x,y)}=[u]_{s,p(x,y)}(\Omega)=\inf\left\{\lambda>0:\int_{\Omega} \int_{\Omega} 
	\frac{|u(x)-u(y)|^{p(x,y)}}{\lambda^{p(x,y)}|x-y|^{N+sp(x,y)}}\,dxdy<1\right\}.
	$$
 The space $(W,\|.\|_{W})$ is a separable reflexive \cite[Lemma 3.1]{Bahrouni1}.\\
 
\noindent In classical Sobolev space theory, embedding results play a fundamental role, particularly in functional analysis and the study of fractional Sobolev spaces. These results are essential for establishing the existence, uniqueness, and regularity of solutions to nonlocal problems, providing a rigorous framework for analyzing the qualitative behavior of solutions and ensuring their well-posedness. In this context, we present our first main result : an embedding theorem that serves as a key tool for addressing such nonlocal problems.
 	\begin{theorem}\label{bb}
 		Let $\Omega\subset \mathbb{R}^{N}$ be an unbounded domain with cone condition and $s\in (0,1)$. Let $p:\overline{\Omega}\times \overline{\Omega}\longrightarrow(1,+\infty)$ be an uniformly continuous function satisfying $(\ref{n})$ and $(\ref{l})$ with $sp^{+}<N$. Assume that $q:\overline{\Omega}\longrightarrow(1,+\infty)$ is an uniformly continuous function such that 
 		\begin{equation}\label{1.3}
 		\overline{p}(x)\leq q(x)\qquad\text{ for all }x\in \overline{\Omega}
 		\end{equation}
 		and
 		\begin{equation}\label{1.4}
 		\inf_{x\in \overline{\Omega}}\left(p_{s}^{*}(x)-q(x)\right)>0.
 		\end{equation}
 		Then $W^{s,p(x,y)}(\Omega)$ is continuously embedded in $L^{q(x)}(\Omega)$.\\
 	\end{theorem}
Based on this embedding result,  we aim to investigate the following nonlocal problem associated with the fractional $p(x,.)$-Laplacian :
	$$\label{P}
	(\mathcal{P}_{s})~~\left \{ \begin{array} {rll}
		\left(-\Delta\right)_{p(x,.)}^{s}u(x)+b(x)|u(x)|^{\overline{p}(x)-2}\,u(x)&=f(x,u)   & \text{in } \Omega,\\
		u &=0 &  \text{in } \mathbb{R}^{N}\backslash\Omega,
	\end{array} \right. 
	$$
	where :
	\begin{itemize}
	\item  $\Omega$ is an unbounded domain in $\mathbb{R}^{N}$ with cone property.
	\item $0<b_{0}\leq b(x)\in L^{\infty}(\Omega)$.
	\item $p:\overline{\Omega}\times\overline{\Omega}\longrightarrow (1,+\infty)$ is  an uniformly continuous function satisfying \eqref{n}-\eqref{1.4}.
	\item The fractional $p(x,.)$-Laplacian operator is given by
			$$
				\left(-\Delta\right)_{p(x,.)}^{s}u(x)=v.p.\int_{\mathbb{R}^{N}} \frac{|u(x)-u(y)|^{p(x,y)-2}(u(x)-u(y))}{|x-y|^{N+sp(x,y)}}\,dy,\quad \forall x \in \mathbb{R}^{N},
			$$
			where $p.v.$ is a commonly used abbreviation  in the principal value sense.
  
	\item The nonlinearity $f:\Omega\times \mathbb{R}\longrightarrow\mathbb{R}$ is a Caratheodory function such that :
	\begin{center}
	$\hspace*{-3cm}(H_{1}):\label{H1}$  $\qquad\quad|f(x,t)|\leq g(x)|t|^{q(x)-1} \quad $ for all $(x,t) \in \Omega\times \mathbb{R}$,
	\end{center} with $q:\overline{\Omega}\longrightarrow(1,+\infty)$ is a bounded uniformly continuous function such that 
	$$
	1<p^{+}<q^{-}=\inf_{x\in \overline{\Omega}} q(x)<p_{s}^{*}(x) \quad \text{for all }x\in \overline{\Omega}
	$$
	and $g \in L^{r(x)}(\Omega)\cap L^{\infty}(\Omega)$ and $g(x)\geq 0$, where $r:\overline{\Omega}\longrightarrow(1,+\infty)$ satisfies
	$$
	\overline{p}(x)\leq h(x)=\frac{r(x)q(x)}{r(x)-1}<p_{s}^{*}(x), \quad \text{for all }x\in \overline{\Omega}.
	$$
	$(H_{2}):\label{H2}$ There exists a positive constant $\mu>p^{+}$ such that 
	$$
	0<\mu F(x,t)\leq tf(x,t) \quad \text{for all }x\in\Omega,~~ t\neq 0,
	$$ 
	where $F(x,t)=\int_{0}^{t}f(x,\tau)\,d\tau \quad \text{for all }(x,t)\in \Omega\times \mathbb{R}$.\\
	\end{itemize}		
In the literature, a significant amount of research has been conducted on the local version of problem \hyperref[P]{($\mathcal{P}_s$)}, specifically on $p(x)$-Laplacian problems in unbounded domains. A notable contribution is the work of X. Fan and X. Han \cite{Fan}, in which the following problem has been studied:
	$$
	(\mathcal{P}_{1})~~\left \{ \begin{array} {rcl}
		-div\left(|\nabla u|^{p(x)-2}\nabla u\right)+|u|^{p(x)-2}u&=&f(x,u) \quad  \text{in } \mathbb{R}^{N},\\
		u\in W^{1,p(x)}(\mathbb{R}^{N}),
	\end{array} \right. 
	$$
Using variational methods, they  investigated the existence and multiplicity of solutions of problem $	(\mathcal{P}_{1})$, under same suitable conditions on $p$ and $f$.\\
\noindent   In this context, F. Yongqiang in \cite{Yong}  establishes the existence of nontrivial weak solutions to the following $p(x)$-Laplacian problem using the mountain pass theorem
	$$
	(\mathcal{P}_{2})\left \{ \begin{array} {rcll}
		-div\left(a(x)|\nabla u|^{p(x)-2}\nabla u\right)+b(x)|u|^{p(x)-2}u&= & f(x,u) &  \text{in } \Omega,\\
		u &=&0 &  \text{in }\partial\Omega,
	\end{array} \right. 
	$$
	where $\Omega \subset \mathbb{R}^{N}$ is an exterieur domain, $p:\overline{\Omega}\longrightarrow(1,+\infty)$ is a bounded Lipshitz continuous function and $f(x,u)=g(x)|u|^{\alpha(x)}$ with $\alpha(\cdot)$ is a variable exponent with special assumptions. \\
	
	\noindent In the fractional scenario, the authors in \cite{ABS0}, deal with  the next class of fractional $p(x,\cdot)$-Kirchhoff-type problems in $\mathbb{R}^N$
	$$
	\left(\mathcal{P}_3\right)~~\left\{\begin{array}{l}
	M\left(\int_{\mathbb{R}^N \times \mathbb{R}^N} \tfrac{1}{p(x, y)} \tfrac{|u(x)-u(y)|^{p(x, y)}}{|x-y|^{N+s p(x, y)}} \mathrm{d} x \mathrm{~d} y\right)\left(-\Delta\right)^s_{p(x,\cdot)} u(x) 
	+|u|^{\bar{p}(x)-2} u=f(x, u) ~ \text { in } \mathbb{R}^N \\
	u \in W^{s, p(x, y)}\left(\mathbb{R}^N\right).
	\end{array}\right.
	$$
	They establish a compact embedding of the space $W^{s, p(x, y)}\left(\mathbb{R}^N\right)$ into $L_{w(x)}^{q(x)}\left(\mathbb{R}^N\right)$. Using this embedding result and some critical-points theorems, they obtain the existence and multiplicity results for the aforementioned problem. \\

\noindent Inspired by these results and aiming to treat related nonlocal problems, we now turn our attention to the analysis of problem \hyperref[P]{($\mathcal{P}_s$)}, for which we define the notion of weak solution as follows :  
		We say that $u\in X_{0}$ is a weak solution of problem \hyperref[P]{($\mathcal{P}_s$)} if
		\begin{align*}
		&\int_{Q}\frac{|u(x)-u(y)|^{p(x,y)-2}(u(x)-u(y))(\phi(x)-\phi(y))}{|x-y|^{N+sp(x,y)}}\,dxdy+\int_{\Omega}b(x)|u(x)|^{\overline{p}(x)-2}u(x)\,\phi(x)\,dx\\
		&\hspace{2cm}-\int_{\Omega}f(x,u)\,\phi(x)\,dx=0,\qquad \forall \phi\in X_{0}.
	\end{align*}

 \noindent Hence, the second main result of this paper can be stated as follows.
   \begin{theorem}\label{3}
   	Let $\Omega\subset \mathbb{R}^{N}$ be unbounded domain with cone condition and $s\in (0,1)$. Let $p:\overline{Q}\longrightarrow(1,+\infty)$ be an uniformly continuous function satisfies $(\ref{n})$ and $(\ref{l})$ on $\overline{Q}$ with $sp^{+}<N$. Assume that $f$ satisfies \hyperref[H1]{$(H_{1})$} and \hyperref[H2]{$(H_{2})$}, then problem \hyperref[P]{($\mathcal{P}_s$)} has a nontrivial solution.
   \end{theorem} 
	
\noindent The remainder of this paper is structured as follows. In Section $\ref{21}$, we provide a brief review of key definitions and fundamental properties of generalized Lebesgue spaces $L^{q(x)}$ and fractional Sobolev spaces with variable exponent $W^{s,p(x,y)}$. In Section $\ref{22}$, we establish a continuous embedding theorem for these spaces into variable exponent Lebesgue spaces on unbounded domains. Finally,  in Section $\ref{23}$, we apply the established embedding to prove the existence of nontrivial weak solutions to problem \hyperref[P]{($\mathcal{P}_s$)}.
 \section{Some preliminary results}$\label{21}$
     In this section, we recall some necessary properties of variable exponent spaces. We  begin with  the generalized Lebesgue spaces, for more details, we refer the reader to \cite{ac,ab,KovRak}, and the references therein.\\
     
  \noindent We first give the definition of an open set with cone condition.
  	\begin{definition}$\label{00}$
  		An open $\Omega$ satisfies the cone condition if there exists a
  		finite cone $O$ such that each $x \in \Omega$ is the vertex of a finite cone $O_x$ contained
  		in $\Omega$ and congruent to $O$.\\ Note that $O_x$ need not be obtained from $O$ by parallel
  		translation, but simply by rigid motion.
  	\end{definition}    
   
   \noindent Now, consider the set 
   $$
   C^{+}(\overline{\Omega})=\left\{q\in C(\overline{\Omega})\,:\,q(x)>1\quad \text{ for all } x\in \overline{\Omega}\right\}.
   $$
   For all $q\in C^{+}(\overline{\Omega})$, we define 
   $$
   q^{+}=\sup_{x\in \overline{\Omega}}q(x)\quad \text{and  }q^{-}=\inf_{x\in \overline{\Omega}}q(x).
   $$ 
Given $q\in C^{+}(\overline{\Omega})$, we define the variable exponent Lebesgue spaces as 
	$$
	L^{q(x)}(\Omega)=\left\{u:\Omega\longrightarrow\mathbb{R}\text{ measurable }:\int_{\Omega}|u(x)|^{q(x)}\,dx<+\infty\right\}.
	$$
	This vector space endowed with the Luxemburg norm, which is defined by 
	$$
	\|u\|_{L^{q(x)}(\Omega)}=\inf\left\{\lambda>0:\int_{\Omega}\left|\frac{u(x)}{\lambda}\right|^{q(x)}\,dx\leq 1\right\}
	$$
	is a separable and reflexive Banach space.\\
	
	\noindent Next, let $\hat{q}\in  C^{+}(\overline{\Omega})$ be the conjugate exponent of $q$, that is, $\tfrac{1}{q(x)}+\tfrac{1}{\hat{q}(x)} =1.$ Then we have the following Hölder-inequality:
	\begin{lemma}[Hölder's inequality] If $u\in L^{q(x)}(\Omega)$ and $v\in L^{\hat{q}(x)}(\Omega)$, then
		$$
		\left|\int_{\Omega}u\,v\,dx\right|\leq \left(\tfrac{1}{q^{-}}+\tfrac{1}{\hat{q}^{-}}\right)\|u\|_{L^{q(x)}(\Omega)}\|v\|_{L^{\hat{q}(x)}(\Omega)}\leq 2\|u\|_{L^{q(x)}(\Omega)}\|v\|_{L^{\hat{q}(x)}(\Omega)}.
		$$
	\end{lemma}
 A very important role in manipulation the generalized Lebesgue spaces with variable exponent is played by the modular $\varrho_{q(.)}:L^{q(x)}(\Omega)\longrightarrow\mathbb{R}$ which is defined by 
	$$
	\varrho_{q(.)}(u)=\int_{\Omega}\left|u(x)\right|^{q(x)}\,dx.
	$$
	\begin{proposition}$\label{5}$
	Let $u\in L^{q(x)}(\Omega)$ and $\{u_{k}\}\subset L^{q(x)}(\Omega)$, $k\in \mathbb{N}$, then we have
		\begin{enumerate}
		\item [$i)$] $\|u\|_{L^{q(x)}(\Omega)} >1\, (=1;\,<1) ~~\Longleftrightarrow ~~\varrho_{q(.)}(u)>1\, (=1;\,<1)$,
		\item [$ii)$] $\|u\|_{L^{q(x)}(\Omega)}>1 ~~\Longrightarrow~~ \|u\|_{L^{q(x)}(\Omega)}^{q_{-}} \leq \varrho_{q(.)}(u)\leq \|u\|_{L^{q(x)}(\Omega)}^{q_{+}}$,
		\item [$iii)$] $\|u\|_{L^{q(x)}(\Omega)}<1 ~~\Longrightarrow~~ \|u\|_{L^{q(x)}(\Omega)}^{q_{+}} \leq \varrho_{q(.)}(u)\leq \|u\|_{L^{q(x)}(\Omega)}^{q_{-}}$,
		\item [$iv)$] $\displaystyle\lim_{k \to +\infty}\|u_{k}-u\|_{L^{q(x)}(\Omega)}=0~~\Longleftrightarrow ~~ \displaystyle\lim_{k\rightarrow +\infty} \varrho_{q(.)}(u_{k}-u)=0$.
		\end{enumerate}
	\end{proposition}
	\begin{proposition}[\text{\cite[Corollary  3.3.4]{ac}}] \label{ad}
		Let $\alpha,\beta\in C^{+}(\overline{\Omega})$ such that $\alpha(x)\leq \beta(x) \quad \text{for all }x\in \overline{\Omega}$. then, there exists a positive constant $C=C(\alpha,\beta,\Omega)$ such that 
		$$
		\|u\|_{L^{\alpha(x)}(\Omega)}\leq C\|u\|_{L^{\beta(x)}(\Omega)} \qquad  \text{for all }u\in  L^{\beta(x)}(\Omega).
		$$
	\end{proposition}
	\begin{proposition}[\cite{Fan}]\label{11}
		If $|u|^{q(x)}\in L^{\frac{h(x)}{q(x)}}(\Omega)$, where $q,h \in C^{+}(\overline{\Omega})$, $q(x)\leq h(x)$, then 
		$u\in  L^{h(x)}(\Omega) $ and there is a number 
		$\overline{q}\in [q^{-},q^{+}]$ such that $\||u|^{q(x)}\|_{L^{\frac{h(x)}{q(x)}}(\Omega)}=\left(\|u\|_{L^{h(x)}(\Omega)}\right)^{\overline{q}}$.
	\end{proposition}
Now, we turn to the fractional case of variable exponent Sobolev spaces, as introduced in \cite{ABS2, Rossi, Kaufmann}. While their definition is provided in the introduction, we summarize here some of their fundamental properties. \\
		Let $p:\overline{\Omega}\times\overline{\Omega}\longrightarrow(1,+\infty)$ be a continuous function satisfies \eqref{n} and \eqref{l}, given $s\in (0,1)$. For any $u\in W$, we define the modular $\varrho_{p(.,.)}:W\longrightarrow\mathbb{R}$ by 
		$$
		\varrho_{p(.,.)}(u)=\int_{\Omega}\int_{\Omega}\frac{|u(x)-u(y)|^{p(x,y)}}{|x-y|^{N+sp(x,y)}}\,dxdy+\int_{\Omega}|u(x)|^{\overline{p}(x)}\,dx.
		$$
	The associated norm is expressed as  follows
		$$
		\|u\|_{\varrho_{p(.,.)}}=inf\left\{\lambda>0:\varrho_{p(.,.)}\left(\frac{u}{\lambda}\right)\leq 1\right\}.
		$$
	\begin{remark} \text{}
		\begin{itemize}
			\item [$(i)$] It is easy to see that $	\|.\|_{\varrho_{p(.,.)}}$ is a norm on $W$ which is equivalent to the norm $\|.\|_{W}$,
			\item [$(ii)$] $	\varrho_{p(.,.)}$ also check the results of Proposition $\ref{5}$.
		\end{itemize}
	\end{remark}
  Let $W_{0}$ denote the closure of $ C_{0}^{\infty}(\Omega)$ in $W$, that is, $W_{0}=\overline{C_{0}^{\infty}(\Omega)}^{\|.\|_{W}}$. We could also get the following properties:
	\begin{lemma}[\cite{ABS2,e}]
	Let $p:\overline{\Omega}\times\overline{\Omega}\longrightarrow(1,+\infty)$ be a continuous function and $s\in (0,1)$. For any $u\in W_{0}$, we have
\begin{itemize}
	\item [$i)$] $1\leq [u]_{s,p(x,y)}~~\Longrightarrow ~~ [u]_{s,p(x,y)}^{p^{-}} \leq \int_{\Omega}\int_{\Omega}\tfrac{|u(x)-u(y)|^{p(x,y)}}{|x-y|^{N+sp(x,y)}}\,dxdy\leq [u]_{s,p(x,y)}^{p^{+}}$,
	\item [$ii)$] $ [u]_{s,p(x,y)} \leq 1~~\Longrightarrow ~~  [u]_{s,p(x,y)}^{p^{+}} \leq \int_{\Omega}\int_{\Omega}\tfrac{|u(x)-u(y)|^{p(x,y)}}{|x-y|^{N+sp(x,y)}}\,dxdy\leq [u]_{s,p(x,y)}^{p^{-}}$. 
\end{itemize}
	\end{lemma}
	\begin{theorem}[\cite{ABS2}]\label{ba}
		Let $\Omega$ be a smooth bounded domain in $\mathbb{R}^{N}$ and let $s\in (0,1)$. Let $p:\overline{\Omega}\times \overline{\Omega}\longrightarrow(1,+\infty)$ be a continuous variable exponent with $sp(x,y)<N$ for all $(x,y)\in \overline{\Omega}\times \overline{\Omega}$. Let $(\ref{n})$ and $(\ref{l})$ be satified. Assume that $r:\overline{\Omega}\longrightarrow(1,+\infty)$ is continuous variable exponent such that 
		$$
		p^{*}_{s}(x)=\tfrac{N\overline{p}(x)}{N-s\overline{p}(x)}>r(x)\geq r^{-}=\inf_{x\in \overline{\Omega}}
		r(x)>1,
		$$
		for all $x\in \overline{\Omega}$. Then, there exists a constant $C=C(N,s,p,r,\Omega)>0$ such that, for any $u\in W$, 
		$$
		\|u\|_{L^{r(x)}(\Omega)}\leq C\|u\|_{W}.
		$$
		Thus, the space $W$ is continuously embedded in $L^{r(x)}(\Omega)$ for any $r\in (1,p_{s}^{*})$. Moreover, this embedding is compact.
	\end{theorem}
 The above embedding theorem is the generalization of the following result in the exponent constant  case. 
 	\begin{theorem}[\cite{DiNezza}]\label{24}
 		Let $s\in (0,1)$ and $p\in [1,+\infty)$ such that $sp<N $. Let $\Omega \subset \mathbb{R}^{N}$ be an extension domain for $W^{s,p}$. Then there exists a positive constant $C=C(N,p,s,\Omega)$ such that, for any $u\in W^{s,p}(\Omega)$, we have 
 		$$
 		\|u\|_{L^{q}(\Omega)}\leq C\|u\|_{W^{s,p}(\Omega)},
 		$$
 		for any $q\in [p,p^{*}]$, i.e. the space $W^{s,p}(\Omega)$ is continuously embedded in $ L^{q}(\Omega)$ for any $q\in  [p,p^{*}]$.
 	\end{theorem}
 \noindent 	Additionally, we have an embedding result in the whole space $\mathbb{R}^N$.
	\begin{theorem}[\cite{DiNezza}]\label{bc}
		Let $s\in (0,1)$ and $p\in [1,+\infty)$ such that $sp<N $. Then, there exists a positive constant $C=C(N,p,s)$ such that, for any measurable and compactly supported function $u:\mathbb{R}^{N}\longrightarrow\mathbb{R}$, we have 
		$$
		\|u\|_{L^{p^{*}}(\mathbb{R}^{N})}\leq C\left(\int_{\mathbb{R}^{N}}\int_{\mathbb{R}^{N}}\tfrac{|u(x)-u(y)|^{p}}{|x-y|^{N+sp}}\,dxdy\right)^{\frac{1}{p}},
		$$
		where 
		$$
		p^{*}=p^{*}(N,s)=\tfrac{Np}{N-sp}
		$$
		is the so-called the "fractional critical exponent".
	\end{theorem}
The study some classes of nonlocal problems involving the fractional $p(x,.)$-Laplacian operator with Dirichlet boundary  condition $u=0$ in $\mathbb{R}^{N}\backslash\Omega$ via variational methods, this requires working within a newly adapted form of  fractional Sobolev space with variable exponent defined as follows (see \cite{ABS8}):
$$
X=W^{s, p(x, y)}(Q):=\left\{\begin{array}{c}
u: \mathbb{R}^N \longrightarrow \mathbb{R} \text { measurable such that } u_{\mid \Omega} \in L^{\overline{p}(x)}(\Omega) \text { with } \\
\int_Q \tfrac{|u(x)-u(y)|^{p(x, y)}}{\lambda^{p(x, y)}|x-y|^{s p(x, y)+N}} d x d y<+\infty, \text { for some } \lambda>0
\end{array}\right\},
$$
	where $Q=\mathbb{R}^{2N}\backslash(\Omega^{c}\times\Omega^{c})$ with $\Omega^{c}=\mathbb{R}^{N}\backslash\Omega$ and $p:\overline{Q}\longrightarrow(1,+\infty)$ satifies $(\ref{n})$ and $(\ref{l})$ on $\overline{Q}$. The space $X$ is endowed with the following norm
	$$
	\|u\|_{X}=\|u\|_{L^{\overline{p}(x)}(\Omega)}+[u]_{X},
	$$
	where $[\cdot]_{X}$ is a Gagliardo semi norm with variable exponent defined by 
	$$
	[u]_{X}=[u]_{s,p(x,y)}(Q)=\inf\left\{\lambda>0:\int_{Q}\frac{|u(x)-u(y)|^{p(x,y)}}{\lambda^{p(x,y)}|x-y|^{N+sp(x,y)}}\,dxdy<1\right\}.
	$$
	Similar to the space $(W,\|.\|_{W})$, we have that $(X,\|.\|_{X})$ is a separable reflexive Banach space.
		\begin{definition}
		Let $p:\overline{Q}\longrightarrow(1,+\infty)$ be a continuous function and $s\in (0,1)$. For any $u\in X$, we define the modular $\rho_{_{X}}:X\longrightarrow\mathbb{R}$ by 
		$$
		\rho_{_{X}}(u)=\int_{Q}\frac{|u(x)-u(y)|^{p(x,y)}}{|x-y|^{N+sp(x,y)}}\,dxdy+\int_{\Omega}|u(x)|^{\overline{p}(x)}\,dx
		$$
		and
		$$
		\|u\|_{\rho_{_{X}}}=\inf\left\{\lambda>0:\rho_{_{X}}\left(\frac{u}{\lambda}\right)\leq 1\right\}.
		$$
	\end{definition}
	\begin{remark}
		\item [$i)$] It is easy to see that $	\|.\|_{\rho_{_{X}}}$ is a norm on $X$ which is equivalent to the norm $\|.\|_{X}$,
		\item [$ii)$] $	\rho_{_{X}}$ also check the results of Proposition $\ref{5}$.
	\end{remark}
	\par Now, let $X_{0}$ denotes the following linear subspace of $X$
	$$
	X_{0}=\left\{u\in X:u=0\,\,\, \text{a.e. in } \mathbb{R}^{N}\backslash\Omega\right\},
	$$
	with the norm 
	$$
	\|u\|_{X_{0}}=[u]_{X}=\inf\left\{\lambda>0:\int_{Q}\frac{|u(x)-u(y)|^{p(x,y)}}{\lambda^{p(x,y)}|x-y|^{N+sp(x,y)}}\,dxdy<1\right\}.
	$$
	\begin{lemma}[\cite{ABS8}]
		$\left(X_{0},\|.\|_{X_{0}}\right)$ is a separable, reflexive and uniformly convex Banach space.
	\end{lemma}
	\par The following theorem, gives the relation between $W$ and $X$ and establish  a compact and continuous embedding of $X$ into Lebesgue spaces with variable exponent.
	\begin{theorem}[\cite{ABS8}]\label{eb}
		Let $\Omega$ be a Lipschitz bounded domain in $\mathbb{R}^{N}$ and let $s \in (0,1)$. Let $p:\overline{Q}\longrightarrow(1,+\infty)$ be a continuous function satisfies $(\ref{n})$ and $(\ref{l})$ on $\overline{Q}$ with $sp^{+}<N$. Then the following assertions hold:
	\begin{itemize}
		\item[$(i)$] If $u\in X$, then $u\in W$. Moreover,
		$$
		\|u\|_{W}\leq \|u\|_{X};
		$$
		\item[$(ii)$] If $u\in X_{0}$, then $u\in W^{s,p(x,y)}(\mathbb{R}^{N})$. Moreover,
		$$
		\|u\|_{W}\leq \|u\|_{W^{s,p(x,y)}(\mathbb{R}^{N})} =\|u\|_{X};
		$$
		\item[$(ii)$] Assume that $r:\overline{\Omega}\longrightarrow(1,+\infty)$ is continuous variable exponent such that 
		$$
		p^{*}_{s}(x)=\tfrac{N\overline{p}(x)}{N-s\overline{p}(x)}>r(x)\geq r^{-}=\inf_{x\in \overline{\Omega}}
		r(x)>1,
		$$
		for all $x\in \overline{\Omega}$. Then, there exists a constant $C=C(N,s,p,r,\Omega)>0$ such that, for any $u\in X$, 
		$$
		\|u\|_{L^{r(x)}(\Omega)}\leq C\|u\|_{X}.
		$$
		Thus, the space $X$ is continuously embedded in $L^{r(x)}(\Omega)$ for any $r\in (1,p_{s}^{*})$. Moreover, this embedding is compact.
	\end{itemize}
	\end{theorem}
	\par Let denote by $\mathscr{L}$ the operator associated to the $\left(-\Delta\right)_{p(x,.)}^{s}$ defined as
	 $$
	\mathscr{L}:X_{0}\longrightarrow X_{0}^{*}
	$$
	$$
	\hspace{3cm}u\longrightarrow \mathscr{L}(u):X_{0}\longrightarrow \mathbb{R}
	$$
	$$
	\hspace{7cm} \phi\longrightarrow \left\langle \mathscr{L}(u),\phi\right\rangle ,
	$$
	such that  
	$$
	\left\langle \mathscr{L}(u),\phi\right\rangle =\int_{Q}\frac{|u(x)-u(y)|^{p(x,y)-2}(u(x)-u(y))(\phi(x)-\phi(y))}{|x-y|^{N+sp(x,y)}}\,dxdy,
	$$
where $\left\langle \cdot,\cdot\right\rangle$ denotes the usual duality between $X_{0}$ and its dual space $X_0^*$.
	\begin{lemma}[\cite{Bahrouni1}]
		Assume that $(\ref{n})$ and $(\ref{l})$ are satisfied on $\overline{Q}$ and that $s\in(0,1)$. Then,  following assertions hold:
		\item[$(i)$] $\mathscr{L}$ is bounded and strictly monotone operator,
		\item[$(ii)$] $\mathscr{L}$ is a mapping of type $(S_{+})$, i.e., if $u_{k}\rightharpoonup u$ in $X_{0}$ and $\displaystyle\limsup_{k\to +\infty}\left\langle \mathscr{L}(u_{k})-\mathscr{L}(u),u_{k}-u\right\rangle \leq 0$, then $u_{k}\longrightarrow u$ in $X_{0}$.
		\item[$(iii)$] $\mathscr{L}$ is a homeomorphism.
	\end{lemma}
	\par  Throughout this paper, for simplicity, we will use $c_{i}$ and $C$ to denote positive constant $($the exact value may change from line to line$)$.
		\section{Embeddings results}$\label{22}$
In this section, we establish the proof of a central result of this work, namely the embedding theorem in unbounded domain stated in Theorem~\ref{bb}. This key result sets the foundation for further applications. 

	\begin{proof}
		\par We will deal with theorem $\ref{bb}$ in two cases :\\
$\bullet$	\textbf{\underline{Case 1}} : Let $u\in W^{s,p(x,y)}(\Omega)\cap L^{\infty}(\Omega)$. Since $\Omega\subset\mathbb{R}^{N}$ be an unbounded domain with cone condition, then there exists a sequence $\{B_{k}\}$ of disjoint open balls contained in $\Omega$ and all having the same positive radius $\epsilon>0$.
	\par Since $p,q $ are uniformly continuous on $\overline{\Omega}$, we can choose $\epsilon>0$ small enough and $t\in (0,s)$ such that 
	\begin{equation}\label{bd}
		p_{k}^{-}\leq \overline{p}(x)\leq q(x)\leq q_{k}^{+}\leq (p_{k}^{-})_{t}^{*}=\frac{Np_{k}^{-}}{N-tp_{k}^{-}} \quad \text{on each }B_{k},
	\end{equation}
	where
	$$
		p_{k}^{-}=\inf_{(x,y)\in B_{k}\times B_{k}}p(x,y)\quad \text{and }\quad q_{k}^{+}=\sup_{x\in B_{k}}q(x).
	$$
	\par Let $\{\psi_{k}\}\subset C^{\infty}(\mathbb{R}^{N})\,\,(k=1,2,\dots)$, $0\leq\psi_{k}\leq 1$, $supp(\psi_{k})\subset B_{k}$ and denote $u_{k}(x)=\psi_{k}(x)u(x)$. Hence, we can apply Theorem $\ref{24}$ to obtain the existence of a positif constant $c_{1}=c_{1}(N,p_{k}^{-},t,\epsilon,B_{k})$ such that 
	\begin{equation}
		\|u_{k}\|_{L^{(p_{k}^{-})_{t}^{*}}(B_{k})}\leq c_{1}\left(\|u_{k}\|_{L^{p_{k}^{-}}(B_{k})}+[u_{k}]_{t,p_{k}^{-}}(B_{k})\right).
	\end{equation}
	By $(\ref{bd})$ and proposition $\ref{ad}$, there exists a positif constant $c_{2}$ such that
	\begin{equation}\label{cd}
		\|u_{k}\|_{L^{q(x)}(\Omega)}=\|u_{k}\|_{L^{q(x)}(B_{k})}\leq c_{2}\|u_{k}\|_{L^{(p_{k}^{-})_{t}^{*}}(B_{k})}.
	\end{equation}
	On the other hand, since $|u_{k}(x)|\leq|u(x)| $, then for any $\lambda>0$, we have
	\begin{equation}\label{be}
		\int_{B_{k}}\left|\frac{u_{k}}{\lambda}\right|^{\overline{p}(x)}\,dx\leq\int_{B_{k}}\left|\frac{u}{\lambda}\right|^{\overline{p}(x)}\,dx .
	\end{equation}
	We set
	$$
	\mathcal{A}_{\lambda,B_{k}}(u_{k})=\left\{\lambda>0:\int_{B_{k}}\left|\tfrac{u_{k}}{\lambda}\right|^{\overline{p}(x)}\,dx\leq1\right\}
~~
\text{ and }
~~
	\mathcal{A}_{\lambda,B_{k}}(u)=\left\{\lambda>0:\int_{B_{k}}\left|\tfrac{u}{\lambda}\right|^{\overline{p}(x)}\,dx\leq1\right\}.
	$$
	From $(\ref{be})$, we have that $\mathcal{A}_{\lambda,B_{k}}(u)\subseteq \mathcal{A}_{\lambda,B_{k}}(u_{k})$. Hence
	\begin{equation}\label{bf}
	\|u_{k}\|_{L^{\overline{p}(x)}(B_{k})}\leq\|u\|_{L^{\overline{p}(x)}(B_{k})}\leq\|u\|_{L^{\overline{p}(x)}(\Omega)}	.
	\end{equation}
As $p_{k}^{-}\leq \overline{p}(x),$ for all $x\in \overline{\Omega}$, then by Proposition $\ref{ad}$ there exists a positive constant $c_{3}$ such that 
	\begin{equation}\label{bg}
	\|u_{k}\|_{L^{p_{k}^{-}}(B_{k})}\leq c_{3}\|u_{k}\|_{L^{\overline{p}(x)}(B_{k})}.
	\end{equation} 
Combining $(\ref{bf})$ and $(\ref{bg})$, we get
	\begin{equation}\label{ce}
	\|u_{k}\|_{L^{p_{k}^{-}}(B_{k})}\leq c_{3}\|u\|_{L^{\overline{p}(x)}(\Omega)}.
	\end{equation}
Next, observe that 
	\begin{align*}
		\int_{B_{k}}\int_{B_{k}}\frac{|u_{k}(x)-u_{k}(y)|^{p_{k}^{-}}}{|x-y|^{N+tp_{k}^{-}}}\,dxdy&=\int_{B_{k}}\int_{B_{k}}\frac{|\psi_{k}(x)u(x)-\psi_{k}(y)u(y)|^{p_{k}^{-}}}{|x-y|^{N+tp_{k}^{-}}}\,dxdy\\
		&\leq2^{p_{k}^{-}}\int_{B_{k}}\int_{B_{k}}\frac{|\psi_{k}(y)|^{p_{k}^{-}}|u(x)-u(y)|^{p_{k}^{-}}}{|x-y|^{N+tp_{k}^{-}}}\,dxdy\\
		&\hspace{0.5cm}+2^{p_{k}^{-}}\int_{B_{k}}\int_{B_{k}}\frac{|\psi_{k}(x)-\psi_{k}(y)|^{p_{k}^{-}}}{|x-y|^{N+tp_{k}^{-}}}\,|u(x)|^{p_{k}^{-}}\,dxdy.
	\end{align*}
Since $0\leq \psi_{k}(x)\leq 1$ and $|\psi_{k}(x)-\psi_{k}(y)|\leq \|\nabla \psi_{k}\|_{\infty}|x-y|$  for all $x,y\in \Omega$, we deduce that 
	\begin{align*}	\int_{B_{k}}\int_{B_{k}}\frac{|u_{k}(x)-u_{k}(y)|^{p_{k}^{-}}}{|x-y|^{N+tp_{k}^{-}}}\,dxdy
	&\leq2^{p_{k}^{-}}\int_{B_{k}}\int_{B_{k}}\frac{|u(x)-u(y)|^{p_{k}^{-}}}{|x-y|^{N+tp_{k}^{-}}}\,dxdy\\
	&\hspace{0.5cm}+2^{p_{k}^{-}}\|\nabla \psi_{k}\|_{\infty}^{p_{k}^{-}}\int_{B_{k}}\left(\int_{B_{k}}|x-y|^{-N+(1-t)p_{k}^{-}}\,dy\right)\,|u(x)|^{p_{k}^{-}}\,dx\\
	&\leq c_{4}\left(\int_{B_{k}}\int_{B_{k}}\frac{|u(x)-u(y)|^{p_{k}^{-}}}{|x-y|^{N+tp_{k}^{-}}}\,dxdy+\int_{B_{k}}|u(x)|^{p_{k}^{-}}\,dx\right).
	\end{align*}
	Then
	$$
	[u_{k}]_{t,p_{k}^{-}}(B_{k})\leq c_{5} \left(\|u\|_{L^{p_{k}^{-}}(B_{k})}+[u]_{t,p_{k}^{-}}(B_{k})\right).
	$$
Considering that $p_{k}^{-}\leq \overline{p}(x)$, for all $x\in B_{k}$, then by proposition $\ref{ad}$ there exists a positif constant $c_{6}$ such that 
	\begin{equation}\label{cc}
		\|u\|_{L^{p_{k}^{-}}(B_{k})}\leq c_{6}\|u\|_{L^{\overline{p}(x)}(B_{k})}\leq c_{6}\|u\|_{L^{\overline{p}(x)}(\Omega)}.
	\end{equation} 
It remains to show that there exists a positive constant 
 $c_{7}$ such that 
	\begin{equation}\label{cb}
		[u]_{t,p_{k}^{-}}(B_{k})\leq c_{7}[u]_{s,p(x,y)}.
	\end{equation}
	Indeed, let us set  
	$$
	U(x,y)=\frac{|u(x)-u(y)|}{|x-y|^{s}}.
	$$
Then by proposition $\ref{ad}$, we have 
	\begin{align*}
		[u]_{t,p_{k}^{-}}(B_{k})&=\left(\int_{B_{k}}\int_{B_{k}}\frac{|u(x)-u(y)|^{p_{k}^{-}}}{|x-y|^{N+tp_{k}^{-}}}\,dxdy\right)^{\frac{1}{p_{k}^{-}}}\\
		&=\left(\int_{B_{k}}\int_{B_{k}}\left(\frac{|u(x)-u(y)|}{|x-y|^{s}}\right)^{p_{k}^{-}}\frac{dxdy}{|x-y|^{N+(t-s)p_{k}^{-}}}\right)^{\frac{1}{p_{k}^{-}}}\\
		&=\|U\|_{L^{p_{k}^{-}}(B_{k}\times B_{k},\nu)}\\
		&\leq c_{8}\|U\|_{L^{p(x,y)}(B_{k}\times B_{k},\nu)},
	\end{align*}
	where 
	$$
	d\nu(x,y)=\tfrac{dxdy}{|x-y|^{N+(t-s)p_{k}^{-}}}.
	$$
	Hence, it suffices to prove that 
	\begin{equation}\label{ca}
	\|U\|_{L^{p(x,y)}(B_{k}\times B_{k},\nu)}\leq c_{9}[u]_{s,p(x,y)}.
\end{equation}
In fact,  let $\eta>0$ be such that 
	$$
	\int_{\Omega}\int_{\Omega}\frac{|u(x)-u(y)|^{p(x,y)}}{\eta^{p(x,y)}|x-y|^{N+sp(x,y)}}\,dxdy<1.
	$$
	 Choose 
	$$
	h=\sup\left\{1,\sup_{(x,y)\in \Omega\times\Omega}|x-y|^{s-t}\right\}\quad \text{and  }\quad\overline{\eta}=\eta h.
	$$
	Then
	\begin{align*}
		&\int_{B_{k}}\int_{B_{k}}\left(\frac{|u(x)-u(y)|}{\overline{\eta}|x-y|^{s}}\right)^{p(x,y)}\frac{dxdy}{|x-y|^{N+(t-s)p_{k}^{-}}}\\
		&=\int_{B_{k}}\int_{B_{k}}\frac{|x-y|^{(s-t)p_{k}^{-}}}{h^{p(x,y)}}\,\frac{|u(x)-u(y)|^{p(x,y)}}{\eta^{p(x,y)}|x-y|^{N+sp(x,y)}}\,dxdy\\
		&\leq\int_{B_{k}}\int_{B_{k}}\frac{|u(x)-u(y)|^{p(x,y)}}{\eta^{p(x,y)}|x-y|^{N+sp(x,y)}} \, dxdy<1.
	\end{align*}
 Thus
	$$
	\|U\|_{L^{p(x,y)}(B_{k}\times B_{k},\nu)}\leq\overline{\eta }=\eta h.
	$$
This yields $(\ref{ca})$. and consequently, inequality $(\ref{cb})$ is proved.\\
Therefore, by combining 
 $(\ref{cc})$ and $(\ref{cb})$, we infer that
	\begin{equation}\label{cf}
			[u_{k}]_{t,p_{k}^{-}}(B_{k})\leq c_{5}\left(c_{6}\|u\|_{L^{\overline{p}(x)}\Omega}+c_{7}[u]_{s,p(x,y)}\right).
	\end{equation} 
	Consequently, by $(\ref{cd})$, $(\ref{ce})$ and $(\ref{cf})$, we get
	$$
	\|u_{k}\|_{L^{q(x)}(\Omega)}\leq C\left(\|u\|_{L^{\overline{p}(x)}(\Omega)}+[u]_{s,p(x,y)}\right).
	$$
 Assume that $\|u_{k}\|_{L^{q(x)}(\Omega)}\geq 1$, then, by Proposition $\ref{5}$, we get 
		$$
		\int_{\Omega}|u_{k}|^{q(x)}\,dx\leq C^{q^{+}}\left(\|u\|_{L^{\overline{p}(x)}(\Omega)}+[u]_{s,p(x,y)}\right)^{q^{+}}.
		$$
		Since $u_{k}(x)\longrightarrow u(x)\quad \text{a.e. }x\in \Omega$, by Fatou's lemma, we have 
		\begin{equation}\label{cg}
			\int_{\Omega}|u|^{q(x)}\,dx\leq C^{q^{+}}\left(\|u\|_{L^{\overline{p}(x)}(\Omega)}+[u]_{s,p(x,y)}\right)^{q^{+}}.
		\end{equation}	
$\bullet$	\textbf{\underline{Case 2}} : Let $u\in W^{s,p(x,y)}(\Omega)$,  we will prove that $(\ref{cg})$ is satisfied in this case.\\ Indeed,
	 for $k=1,2,\dots,$ let us consider
	$$
	u_{k}(x)=\left \{ \begin{array} {cl}
		u(x)   &\text{if } |u(x)|\leq k,\\
	k\,\text{sgn}(u(x)) &\text{if } |u(x)|> k.
\end{array} \right.	
	$$
	Then $u_{k} \in W^{s,p(x,y)}(\Omega)\cap L^{\infty}(\Omega)$. Notice that $|u_{k}(x)|\leq |u(x)|$ and
	\begin{equation}\label{ck}
	\frac{|u_{k}(x)-u_{k}(y)|}{|x-y|^{\frac{N}{p(x,y)}+s}}\leq \frac{|u(x)-u(y)|}{|x-y|^{\frac{N}{p(x,y)}+s}}.
	\end{equation}
 From $(\ref{cg})$, we get 
	$$
	\int_{\Omega}|u_{k}|^{q(x)}\,dx\leq C^{q^{+}}\left(\|u_{k}\|_{L^{\overline{p}(x)}(\Omega)}+[u_{k}]_{s,p(x,y)}\right)^{q^{+}}.
	$$
 Since $|u_{k}(x)|\leq |u(x)|$, Similarly to $(\ref{bf})$, we obtain
	\begin{equation}\label{da}
		\|u_{k}\|_{L^{\overline{p}(x)}(\Omega)}\leq \|u\|_{L^{\overline{p}(x)}(\Omega)}.
	\end{equation}
 Now, we prove that 
    $$
     [u_{k}]_{s,p(x,y)}\leq [u]_{s,p(x,y)} .
	$$
	Indeed, by $(\ref{ck})$, for all $\gamma>0$, we have 
	$$
	\int_{\Omega}\int_{\Omega}\tfrac{|u_{k}(x)-u_{k}(y)|^{p(x,y)}}{\gamma^{p(x,y)}|x-y|^{N+sp(x,y)}} \leq \int_{\Omega}\int_{\Omega}\tfrac{|u(x)-u(y)|^{p(x,y)}}{\gamma^{p(x,y)}|x-y|^{N+sp(x,y)}}.
	$$ 
We set 
	$$
	\mathcal{B}_{\lambda,\Omega\times\Omega}^{s}(u_{k})=\left\lbrace \gamma>0:\int_{\Omega}\int_{\Omega}\tfrac{|u_{k}(x)-u_{k}(y)|^{p(x,y)}}{\gamma^{p(x,y)}|x-y|^{N+sp(x,y)}}\leq 1\right\rbrace 
	$$
	and
	$$
	\mathcal{B}_{\lambda,\Omega\times\Omega}^{s}(u)=\left\lbrace \gamma>0:\int_{\Omega}\int_{\Omega}\tfrac{|u(x)-u(y)|^{p(x,y)}}{\gamma^{p(x,y)}|x-y|^{N+sp(x,y)}}\leq 1\right\rbrace .
	$$
	 Using $(\ref{ck})$, we have $\mathcal{B}_{\lambda,\Omega\times\Omega}^{s}(u)\subseteq \mathcal{B}_{\lambda,\Omega\times\Omega}^{s}(u_{k})$. Hence, 
	 \begin{equation}\label{db}
	 	[u_{k}]_{s,p(x,y)}\leq [u]_{s,p(x,y)}.
	 \end{equation} 
	Combining $(\ref{da})$ and $(\ref{db})$, we deduce that 
	 $$
	 \int_{\Omega}|u_{k}|^{q(x)}\,dx\leq C^{q^{+}}\left(\|u\|_{L^{\overline{p}(x)}(\Omega)}+[u]_{s,p(x,y)}\right)^{q^{+}},
	 $$
	 since $u_{k}(x)\longrightarrow u(x)\quad \text{a.e. }x\in \Omega$, by Fatou's lemma, we have
	 	$$
	 	\int_{\Omega}|u|^{q(x)}\,dx\leq C^{q^{+}}\left(\|u\|_{L^{\overline{p}(x)}(\Omega)}+[u]_{s,p(x,y)}\right)^{q^{+}}<+\infty.
	 $$
	 Thus $u\in L^{q(x)}(\Omega)$. This means that $W^{s,p(x,y)}(\Omega)$ is continuously embedded in $L^{q(x)}(\Omega)$. \\
	 The proof is complete.
	\end{proof}
	\begin{remark}
	The obtained result in	Theorem $\ref{bb}$ remains true if we replace $W^{s,p(x,y)}(\Omega)$ by $W^{s,p(x,y)}_{0}(\Omega)$.
	\end{remark}
	\par In the following theorem, we establish  the continuous embedding of $X$ into Lebesgue spaces with variable exponent in unbounded domain.
		\begin{theorem}$\label{ea}$
		Let $\Omega\subset \mathbb{R}^{N}$ be an unbounded domain with cone condition and $s\in (0,1)$. Let $p:\overline{Q}\longrightarrow(1,+\infty)$ be an uniformly continuous function satisfied $(\ref{n})$ and $(\ref{l})$ on $\overline{Q}$ with $sp^{+}<N$. Assume that $q:\overline{\Omega}\longrightarrow(1,+\infty)$ is an uniformly continuous function such that 
		$$
		\overline{p}(x)\leq q(x)\qquad\text{ for all }x\in \overline{\Omega}
		$$
		and
		$$
		\inf_{x\in \overline{\Omega}}\left(p_{s}^{*}(x)-q(x)\right)>0,
		$$
		then the space $X$ is continuously embedded in $L^{q(x)}(\Omega)$
	\end{theorem}
	\begin{proof}
		Combining Theorem \eqref{eb}-$(i)$ with Theorem $\ref{bb}$, we obtain that the space $X$ is continuously embedded in $L^{q(x)}(\Omega)$.
	\end{proof}
	\begin{remark}$\label{dc}$
		Theorem $\ref{ea}$ remains true if we replace the space $X$ by $X_{0}$.
	\end{remark}
		\section{Existence Results}$\label{23}$
 In this section, we will establish the proof of the main existence result stated in Theorem \ref{3} by using the mountain pass theorem. To this end, wet  consider the energy functional $\mathcal{J}:X_{0}\longrightarrow\mathbb{R}$ associated to problem \hyperref[P]{($\mathcal{P}_s$)}, defined by  
	$$
	\mathcal{J}(u)=\int_{Q}\frac{1}{p(x,y)}\,\frac{|u(x)-u(y)|^{p(x,y)}}{|x-y|^{N+sp(x,y)}}\,dxdy+\int_{\Omega}b(x)\,\frac{|u(x)|^{\overline{p}(x)}}{\overline{p}(x)}dx-\int_{\Omega}F(x,u)\,dx.
	$$
	We set $$
	\mathcal{K}(u)=\int_{\Omega}F(x,u)\,dx.
	$$
	The fonctional $\mathcal{K}$ has the following basic properties.
	\begin{lemma}$\label{2}$ 	Under the same assumptions of theorem $\ref{3}$ and
		suppose that \hyperref[H1]{$(H_{1})$} is verified. Then,
	\begin{itemize}
		\item [$(i)$] $\mathcal{K}$ is strongly continuous on $X_{0}$.
		\item [$(ii)$] $\mathcal{K}\in \mathscr{C}^{1}(X_{0},\mathbb{R})$ and
		$$
		\left\langle \mathcal{K}^{'}(u),\phi\right\rangle =\int_{\Omega}f(x,u)\phi\,dx,\quad \text{for all }u,\phi\in X_{0}.
		$$
	\end{itemize}
	\end{lemma}
	\begin{proof}
     $(i)$ Let $\Omega_{k}=\{x\in\Omega:|x|<k\}$. By \hyperref[H1]{$(H_{1})$}, $\mathcal{K}$ is well defined. Assume that $u_{n}\rightharpoonup u $ in $X_{0}$. We have
		\begin{equation}\label{dd}
			\left|\mathcal{K}(u_{n})-\mathcal{K}(u)\right|\leq \int_{\Omega_{k}}|F(x,u_{n})-F(x,u)|\,dx+\frac{1}{q_{-}}\int_{\Omega\backslash\Omega_{k}}g(x)\left(|u_{n}(x)|^{q(x)}-|u(x)|^{q(x)}\right)\,dx.
		\end{equation}
		Since $ \overline{p}(x)\leq h(x)<p_{s}^{*}(x)~$, for all $x\in \overline{\Omega}$, then by Theorem $\ref{ea}$ and Remark $\ref{dc}$, $X_{0}$ is continuously  embedded in $L^{h(x)}(\Omega)$. Thus for all $u\in X_{0}$, we have $u\in L^{h(x)}(\Omega) $. Then one has 
		$$
		\int_{\Omega}\left(|u|^{q(x)}\right)^{\frac{h(x)}{q(x)}}\,dx=\int_{\Omega}|u|^{h(x)}\,dx<+\infty.
		$$ 
      Hence $|u|^{q(x)}\in L^{\frac{h(x)}{q(x)}}(\Omega)$. Since $\{u_{n}\}$ is bounded in $X_{0}$, then $\{u_{n}\}$ is bounded in $L^{h(x)}(\Omega)$. So, there exists a psitive constant $C$ such that 
      \begin{equation}\label{de}
      \max\left\{\||u_{n}|^{q(x)}\|_{\frac{h(x)}{q(x)}},\||u|^{q(x)}\|_{\frac{h(x)}{q(x)}}\right\}\leq C.
    \end{equation}
      Since $g\in L^{r(x)}(\Omega)$. Then
     \begin{equation}\label{10}
      	\|g\|_{L^{r(x)}(\Omega\backslash\Omega_{k})}\longrightarrow 0\quad \text{as }\quad k\longrightarrow +\infty.
    \end{equation}
      Then for all $\epsilon>0$ and $k$ large enough, we get 
      \begin{equation}\label{df}
      	\|g\|_{L^{r(x)}(\Omega\backslash\Omega_{k})}\leq \frac{q^{-}\epsilon}{8C}.
      \end{equation}
       As  $\quad q(x)<p_{s}^{*}(x) \quad $  for all  $x\in\overline{\Omega}$, then, by Theorem $\ref{ba}$,  the embedding $W^{s,p(x,y)}(\Omega_{k})\hookrightarrow L^{q(x)}(\Omega_{k})$ is compact. Hence $u_{n}\rightharpoonup u$ in $W^{s,p(x,y)}(\Omega_{k})$ implies that there exists a subsequence still denote by $\{u_{n}\}$ such that $\quad u_{n}\longrightarrow u$ in $L^{q(x)}(\Omega_{k})$. From $(H_{1})$,we get
      $$
      |F(x,u_{n})|\leq \frac{1}{q^{-}}\,g(x)|u_{n}|^{q(x)}.
      $$
      By the dominated convergence theorem, we obtain 
      $$
      \int_{\Omega_{k}}|F(x,u_{n})-F(x,u)|\,dx\longrightarrow 0\quad \text{as}\quad n\longrightarrow +\infty.
      $$
      Hence, there exists $n_{1}>0$ such that for any $n\geq n_{1}$, one has 
      \begin{equation}\label{dg}
      	\int_{\Omega_{k}}|F(x,u_{n})-F(x,u)|\,dx\leq \frac{\epsilon}{2}.
      \end{equation}
      Combining $(\ref{dd})-(\ref{dg})$ with Hölder inequality, we get
      \begin{align*}
      	\left|\mathcal{K}(u_{n})-\mathcal{K}(u)\right|&\leq\frac{\epsilon}{2}+\frac{2}{q^{-}}\|g\|_{L^{r(x)}(\Omega\backslash\Omega_{k})}\left(\||u_{n}|^{q(x)}\|_{L^{\frac{h(x)}{q(x)}}(\Omega\backslash\Omega_{k})}+\||u|^{q(x)}\|_{L^{\frac{h(x)}{q(x)}}(\Omega\backslash\Omega_{k})}\right)\\
      	&\leq \frac{\epsilon}{2}+\frac{4C}{q^{-}}\,\frac{q^{-}\epsilon}{8C}=\epsilon.
      \end{align*}
      We deduce that 
      $$
      \mathcal{K}(u_{n})\longrightarrow \mathcal{K}(u)\quad \text{as}\quad n\longrightarrow+\infty.
      $$
      $(ii)$ For any $u,\phi \in X_{0}$, we have 
      \begin{equation}	\label{di}
       \left\langle \mathcal{K}^{'}(u),\phi\right\rangle =\lim_{t \to 0}\frac{\mathcal{K}(u+t\phi)-\mathcal{K}(u)}{t}
      =\lim_{t \to 0}\int_{\Omega}\frac{F(x,u+t\phi)-F(x,u)}{t}\,dx.
      \end{equation}
      Let us consider 
     $$\begin{array}{rl} M:[0,1]&\longrightarrow\mathbb{R}\\
  \alpha &\longmapsto\frac{F(x,u+t\alpha\phi)-F(x,u)}{t}.
     \end{array}$$
      The function $M$ is continuous on $[0,1]$, and differentiable on $(0,1)$. Then, by the mean value theorem, there exists $\theta\in (0,1)$ such that 
      $$
      M^{'}(\theta)=M(1)-M(0).
      $$
     That is,
      \begin{equation}\label{dj}
      	f(x,u+t\theta\phi)\phi=\frac{F(x,u+t\phi)-F(x,u)}{t}.
      \end{equation}
      Combining $(\ref{di})$ and $(\ref{dj})$, we get 
      $$
      \left\langle \mathcal{K}^{'}(u),\phi\right\rangle =\lim_{t \to 0}\int_{\Omega}f(x,u+t\theta\phi)\phi\,dx.
      $$
      Since $\,\, t,\theta\in[0,1]$, so $t\theta\leq 1$, which implies
      $$
      |f(x,u+t\theta\phi)\phi|\leq 2^{q^{+}-1}\left(g(x)|u|^{q(x)-1}\phi+g(x)|\phi|^{q(x)}\right).
      $$
      By Hölder inequality, we have
      $$
      \int_{\Omega}g(x)|\phi|^{q(x)}\,dx\leq 2\|g\|_{L^{r(x)}(\Omega)}\||\phi|^{q(x)}\|_{L^{\frac{h(x)}{q(x)}}}<+\infty,
      $$
      where $\tfrac{1}{r(x)}+\tfrac{q(x)}{h(x)}=1$.\\
     Next, we claim that 
      \begin{equation}\label{1}
      \int_{\Omega}g(x)|u|^{q(x)-1}|\phi|\,dx\leq 2\|g|u|^{q(x)-1}\|_{L^{\hat{h}(x)}(\Omega)}\|\phi\|_{L^{h(x)}(\Omega)}<+\infty,
      \end{equation}
      where $\tfrac{1}{h(x)}+\tfrac{1}{\hat{h}(x)}=1$. Indeed, by Hölder inequality, we obtain
      $$
      \int_{\Omega}\left(g(x)|u|^{q(x)-1}\right)^{\hat{h}(x)}dx\leq 2\||g|^{\hat{h}(x)}\|_{L^{\frac{r(x)}{\hat{h}(x)}}(\Omega)}\||u|^{(q(x)-1)\hat{h}(x)}\|_{L^{w(x)}(\Omega)},
      $$
      where $\tfrac{1}{w(x)}+\tfrac{\hat{h}(x)}{r(x)}=1$. One has 
     $$
      \int_{\Omega}\left(|g|^{\hat{h}(x)}\right)^{\frac{r(x)}{\hat{h}(x)}}\,dx=\int_{\Omega}|g|^{r(x)}\,dx<+\infty \quad \text{and }\quad \int_{\Omega}|u|^{(q(x)-1)w(x)\hat{h}(x)}\,dx=\int_{\Omega}|u|^{h(x)}\,dx<+\infty.      $$
      Then $(\ref{1})$ hold true. On the other hand, since $f$ is a Carathéodory function, we have
      $$
      \lim_{t \to 0} f(x,u+t\theta\phi)\phi=f(x,u)\phi.
      $$
      Hence, by the dominated convergence theorem, we deduce that
      $$
      \left\langle \mathcal{K}^{'}(u),\phi\right\rangle =\int_{\Omega} f(x,u)\phi\,dx \quad \text{for all }u,\phi \in X_{0}.
      $$
      \par Next, Suppose that $\quad u_{n}\longrightarrow u \text{ in }X_{0}$, and we show that $\quad \mathcal{K}^{'}(u_{n})\longrightarrow\mathcal{K}^{'}(u) \text{ in }X_{0}^{*}$.\\
      Again, by Hölder inequality, we have 
    \begin{align*}
    	\left| \left\langle \mathcal{K}^{'}(u_{n})-\mathcal{K}^{'}(u),\phi\right\rangle \right| &=\int_{\Omega}\left|f(x,u_{n})-f(x,u)\right||\phi|\,dx\\
    	&\leq 2\|f(.,u_{n})-f(.,u)\|_{L^{\hat{q}(x)}(\Omega)}\|\phi\|_{L^{q(x)}(\Omega)},
    \end{align*}
    where $\tfrac{1}{q(x)}+\tfrac{1}{\hat{q}(x)}=1$. Then
    $$
    \|\mathcal{K}^{'}(u_{n})-\mathcal{K}^{'}(u)\|_{X_{0}^{*}}\leq 2 \|f(\cdot,u_{n})-f(\cdot,u)\|_{L^{\hat{q}(x)}(\Omega)}.
    $$
    Since$\quad u_{n}\longrightarrow u \text{ in }X_{0}$, then $\quad u_{n}\longrightarrow u \text{ in } L^{q(x)}(\Omega)$. Hence, there exists a subsequence still denote by $\{u_{n}\}$ such that$ \quad u_{n}(x)\longrightarrow u(x) \text{ a.e. in }\Omega$ and there exists $v $ in $L^{q(x)}(\Omega)$ such that $ |u_{n}(x)|\leq v(x)$. Then, we have
    $$
    f(x,u_{n}(x))\longrightarrow f(x,u(x)) \quad \text{a.e. in }\Omega \quad \text{and }\,\,|f(x,u_{n}(x))|\leq g(x)|v(x)|^{q(x)-1}.
    $$
    Thus, by the dominated convergence theorem, we deduce that 
    $$
    f(\cdot,u_{n})\longrightarrow f(\cdot ,u) \quad \text{ in }L^{\hat{q}(x)}(\Omega)\quad \text{as }n\longrightarrow +\infty.
    $$
    Consequently
    $$
    \mathcal{K}^{'}(u_{n})\longrightarrow \mathcal{K}^{'}(u)\quad \text{in }X_{0}^{*}.
    $$
    Which complete the proof. \\
	\end{proof}
\noindent By the same argument as in the proof of Lemma $\ref{2}-(ii)$, we have that $\mathcal{J}\in \mathscr{C}^{1}(X_{0},\mathbb{R})	$ is well defined and its G$\hat{a}$teau derivative is given by
	\begin{align*}	
	\left\langle \mathcal{J}^{'}(u),\phi\right\rangle &=\int_{Q}\tfrac{|u(x)-u(y)|^{p(x,y)-2}(u(x)-u(y))(\phi(x)-\phi(y))}{|x-y|^{N+sp(x,y)}}\,dxdy+\int_{\Omega}b(x)|u|^{\overline{p}(x)-2}u(x)\phi(x)\,dx\\
	&\hspace{4cm}-\int_{\Omega}f(x,u)\phi\,dx,
 \end{align*} 
   for all $u,\phi\in X_{0}$. Thus, the weak solution of $(\mathcal{P}_{s})$ corresponds to the critical points of $\mathcal{J}$.\\
   
 \noindent Let us defined the operator $\widetilde{\mathscr{L}}$ as follows 
    $$
   \widetilde{\mathscr{L}}:X\longrightarrow X^{*}
   $$
   $$
   \hspace{3cm}u\longrightarrow \widetilde{\mathscr{L}}(u):X\longrightarrow \mathbb{R}
   $$
   $$
   \hspace{7cm} \phi\longrightarrow \left\langle \widetilde{\mathscr{L}}(u),\phi\right\rangle 
   $$
   such that  
   $$
   \left\langle \widetilde{\mathscr{L}}(u),\phi\right\rangle =\int_{Q}\tfrac{|u(x)-u(y)|^{p(x,y)-2}(u(x)-u(y))(\phi(x)-\phi(y))}{|x-y|^{n+sp(x,y)}}\,dxdy+\int_{\Omega}b(x)|u|^{\overline{p}(x)-2}u(x)\phi(x)dx,
   $$
 where $\left\langle \cdot,\cdot\right\rangle$ denotes the usual duality between $X_{0}$ and its dual space $X_0^*$.
   \begin{lemma}$\label{19}$
   	Under the same assumptions of theorem $\ref{3}$. Then, following assertions hold: \begin{itemize}
   	\item[$(i)$] $\widetilde{\mathscr{L}}$ is bounded and strictly monotone operator,
   	\item[$(ii)$] $\widetilde{\mathscr{L}}$ is of type $(S_{+})$, that is, if $u_{k}\rightharpoonup u$ in $X_{0}$ and $\displaystyle\limsup_{k\to +\infty}\left\langle \widetilde{\mathscr{L}}(u_{k})-\widetilde{\mathscr{L}}(u),u_{k}-u\right\rangle \leq 0$, then $u_{k}\longrightarrow u$ in $X_{0}$.
   	\item[$(iii)$] $\widetilde{\mathscr{L}}$ is a homeomorphism. 	
   	   	\end{itemize}
   \end{lemma}
   \begin{proof}
   	Similar to proof of  \cite[Lemma $4.2$]{Bahrouni1}.
   \end{proof}
 Let us first verify that the functional $J$ satisfies the mountain pass geometry.
   \begin{lemma}\label{l43}
   	Under the same assumptions of theorem $\ref{3}$, there exists $R>0$ and $a>0$ such that 
   	$$
   	\mathcal{J}(u)\geq a>0 \quad \text{for any }u\in X_{0}\text{ with }\|u\|_{X_{0}}=R.
   	$$
   \end{lemma}
   \begin{proof}
   	For any $u\in X_{0}$, we have 
    \begin{align*}
    \mathcal{J}(u)&=\int_{Q}\frac{1}{p(x,y)}\,\frac{|u(x)-u(y)|^{p(x,y)}}{|x-y|^{N+sp(x,y)}}\,dxdy+\int_{\Omega}b(x)\,\frac{|u(x)|^{\overline{p}(x)}}{\overline{p}(x)}dx-\int_{\Omega}F(x,u)\,dx\\
    &\geq \frac{\min\{1,b_{0}\}}{p^{+}}\rho_{_{X}}(u)-\int_{\Omega}F(x,u)\,dx.
\end{align*}
   By $(H_{1})$, we have
   \begin{equation}\label{6}
   |F(x,t)|\leq \frac{\|g\|_{\infty}}{q^{-}}|t|^{q(x)}\quad \text{for all }(x,t)\in \Omega\times \mathbb{R}.
\end{equation}
   Since$\quad \overline{p}(x)<q(x)<p_{s}^{*}(x)$ for all $x\in \overline{\Omega}$, then, by Remark $\ref{dc}$, we have that $X_{0}$ is continuously embedded in $L^{q(x)}(\Omega)$, that is, there exists a positive constant $c_{8}$ such that 
   \begin{equation}\label{4}
   	\|u\|_{L^{q(x)}(\Omega)}\leq c_{8}\|u\|_{X_{0}}, \quad \text{for all }u\in X_{0}.
   \end{equation}
   We fix $R\in(0,1)$ for which $R<\dfrac{1}{c_{8}}$. Then, relation $(\ref{4})$ implies 
   $$
   \|u\|_{L^{q(x)}(\Omega)}<1 \quad \text{for all }u\in X_{0}\text{ with }R=\|u\|_{X_{0}}.
   $$
   By Proposition $\ref{5}$, we have
   \begin{equation}\label{7}
   	\int_{\Omega}|u(x)|^{q(x)}\,dx\leq \|u\|_{L^{q(x)}(\Omega)}^{q^{-}}\quad \text{for all }u\in X_{0}\text{ with }R=\|u\|_{X_{0}}.
   \end{equation}
   Since $\|u\|_{X_{0}}\leq 1$, combining $(\ref{6})-(\ref{7})$, we deduce that for all $u\in X_{0}$ with $R=\|u\|_{X_{0}}$
   \begin{align*}
   	\mathcal{J}(u)&\geq \frac{\min\{1,b_{0}\}}{p^{+}}\|u\|_{X_{0}}^{p^{+}}-\frac{\|g\|_{\infty}}{q^{-}}\int_{\Omega}|u|^{q(x)}\,dx\\
   	&\geq\frac{\min\{1,b_{0}\}}{p^{+}}\|u\|_{X_{0}}^{p^{+}}-\frac{\|g\|_{\infty}}{q^{-}}c_{8}^{q^{-}}\|u\|_{X_{0}}^{q^{-}}\\
   	&\geq\frac{\min\{1,b_{0}\}}{p^{+}}R^{p^{+}}-\frac{\|g\|_{\infty}}{q^{-}}c_{8}^{q^{-}}R^{q^{-}}\\
   	&\geq R^{p^{+}}\left(\frac{\min\{1,b_{0}\}}{p^{+}}-\frac{\|g\|_{\infty}}{q^{-}}c_{8}^{q^{-}}R^{q^{-}-p^{+}}\right).
   \end{align*}
   As $p^{+}<q^{-}$, then for any $u\in X_{0}$ with $\|u\|_{X_{0}}=R<\min\left\{\dfrac{1}{c_{8}},\left(\dfrac{\min\{1,b_{0}\}}{2p^{+}}.\dfrac{q^{-}}{\|g\|_{\infty}c_{8}^{q^{-}}}\right)^{\frac{1}{q^{-}-p^{+}}}\right\}$, there exists $a=\dfrac{\min\{1,b_{0}\}}{2p^{+}}R^{p^{+}}$ such that 
   $\mathcal{J}(u)\geq a>0. $
   which  complete the proof.
   \end{proof}
   \begin{lemma}\label{l44}
   	Under the same assumptions of theorem $\ref{3}$, there exists $u_{0}\in X_{0}$ with $\|u_{0}\|_{X_{0}}>R$ such that $\mathcal{J}(u_{0})<0$.
   \end{lemma}
  \begin{proof}
  	For any $u\in X_{0}$, condition \hyperref[H2]{$(H_{2})$} implies that 
  	\begin{equation}\label{8}
  		F(x,tu)\geq c_{9}t^{\mu}|u|^{\mu} \quad \text{for all }t\geq 1. 
  	\end{equation}
  	Let $\phi \in C_{0}^{\infty}(\Omega)$, $\phi \geq 0$ and $\phi\neq 0$. By $(\ref{8})$, for any $t>1$, we get 
  	\begin{align*}
  		\mathcal{J}(t\phi)&=\int_{Q}\frac{t^{p(x,y)}}{p(x,y)}\frac{|\phi(x)-\phi(y)|^{p(x,y)}}{|x-y|^{N+sp(x,y)}}\,dxdy+\int_{\Omega}b(x)\frac{|t\phi|^{\overline{p}(x)}}{\overline{p}(x)}\,dx-\int_{\Omega}F(x,t\phi)\,dx\\
  		&\leq \frac{t^{p^{+}}}{p^{-}}\int_{Q}\frac{|\phi(x)-\phi(y)|^{p(x,y)}}{|x-y|^{N+sp(x,y)}}\,dxdy+\frac{t^{p^{+}}|b|_{\infty}}{p^{-}}\int_{\Omega}|\phi|^{\overline{p}(x)}dx-c_{9}t^{\mu}\int_{\Omega}|\phi|^{\mu}dx\\
  		&\leq\frac{\max\{1,|b|_{\infty}\}t^{p^{+}}}{p^{-}}\rho_{_{X}}(\phi)-c_{9}t^{\mu}\int_{\Omega}|\phi|^{\mu}dx.
  	\end{align*}
  	Since $\mu>p^{+}$, then
  	$$
  	\mathcal{J}(t\phi)\longrightarrow-\infty \quad \text{as }t\longrightarrow +\infty.
  	$$
  	Hence, for $t>1$ large enough, we can take $u_{0}=t\phi $ with $\|u_{0}\|_{X_{0}}>R$ and $\mathcal{J}(u_{0})<0$. This concludes the proof.
  \end{proof}
  \noindent  Now, we are in position to prove Theorem $\ref{3}$\\
  \begin{proof}[Proof of Theorem $\ref{3}$]
  	We need to prove that $\mathcal{J}$ satisfies the Palais-Small condition on $X_{0}$. Let $\{u_{n}\}\subset X_{0}$ be a sequence satisfying 
  	\begin{equation}\label{9}
  	\mathcal{J}(u_{n})\longrightarrow c_{10} \quad \text{and }\quad 
  	\mathcal{J}^{'}(u_{n})\longrightarrow 0 \quad \text{as }n\longrightarrow+\infty. 
  	\end{equation}
  	We first show that $\{u_{n}\}$ is bounded in $X_{0}$. Indeed, we assume the contrary. Then, passing eventually to a subsequence, still denote by $\{u_{n}\}$, we can suppose that $\|u_{n}\|_{X_{0}}\longrightarrow +\infty$ as $n\longrightarrow +\infty$. Thus we can consider that $\|u_{n}\|_{X_{0}}>1$ for all $n$. By $(\ref{9})$, \hyperref[H2]{$(H_{2})$}  and Proposition $\ref{5}$, for $n$ large enough, we get
  	\begin{align*}
  	c_{10}+\|u_{n}\|_{X_{0}}&\geq \mathcal{J}(u_{n})-\frac{1}{\mu}\left\langle \mathcal{J}^{'}(u_{n}),u_{n}\right\rangle \\
  	&\geq\int_{Q}\frac{|u_{n}(x)-u_{n}(y)|^{p(x,y)}}{p(x,y)|x-y|^{N+sp(x,y)}}\,dxdy+\int_{\Omega}b(x)\frac{|u(x)|^{\overline{p}(x)}}{\overline{p}(x)}\,dx-\int_{\Omega}F(x,u_{n})\,dx\\
  	&\hspace{0.3cm}-\frac{1}{\mu}\int_{Q}\frac{|u_{n}(x)-u_{n}(y)|^{p(x,y)}}{|x-y|^{N+sp(x,y)}}\,dxdy-\frac{1}{\mu}\int_{\Omega}b(x)|u_{n}(x)|^{\overline{p}(x)}dx+\frac{1}{\mu}\int_{\Omega}f(x,u_{n})u_{n}\,dx\\
  	&\geq \left(\frac{1}{p^{+}}-\frac{1}{\mu}\right)\min\{1,b_{0}\}\rho_{_{X}}(u_{n})+\int_{\Omega}\left(\frac{1}{\mu}f(x,u_{n})u_{n}-F(x,u_{n})\right)dx\\
  	&\geq \left(\frac{1}{p^{+}}-\frac{1}{\mu}\right)\min\{1,b_{0}\}\|u_{n}\|_{X_{0}}^{p^{-}}.
  	\end{align*}
  	Dividing the above inequality by $\|u_{n}\|_{X_{0}}$, since $1<p^{-}$ and if we passing to the limit as $n\longrightarrow +\infty $, we obtain a contradiction. Then $\{u_{n}\}$ is bounded in $X_{0}$. Hence, there exists a subsequence, still denote by$\{u_{n}\}$, and $u\in X_{0}$ such that $\,\, u_{n}\rightharpoonup u $ in $X_{0}$, then from $(\ref{9})$, we infer that
  	$$
  	\left\langle \mathcal{J}^{'}(u_{n}),u_{n}-u\right\rangle \longrightarrow 0 \quad \text{as }n\longrightarrow +\infty,
  	$$ 
  	that is,
  	$$
  	\int_{Q}\frac{|u_{n}(x)-u_{n}(y)|^{p(x,y)-2}(u_{n}(x)-u_{n}(y))((u_{n}(x)-u_{n}(y))-(u(x)-u(y)))}{|x-y|^{N+sp(x,y)}}\,dxdy
  	$$
  	$$
  	+\int_{\Omega}b(x)|u_{n}|^{\overline{p}(x)-2}u_{n}(u_{n}-u)\,dx-\int_{\Omega}f(x,u_{n})(u_{n}-u)\,dx\longrightarrow 0.
  	$$
  	By $(H_{1})$, we have
  	\begin{align*}
  		\left|\int_{\Omega}f(x,u_{n})(u_{n}-u)\,dx \right|&\leq \int_{\Omega}g(x)|u_{n}|^{q(x)-1}|u_{n}-u|dx\\
  		&\leq 2^{q^{-}-1}\left(\int_{\Omega}g(x)|u_{n}-u|^{q(x)}dx+\int_{\Omega}g(x)|u|^{q(x)-1}|u_{n}-u|\,dx\right)\\
  		&\leq 2^{q^{-}-1}\left(I_{1,n}+I_{2,n}\right).
  	\end{align*}
  	Let $\Omega_{k}=\{x\in \Omega:|x|\leq k\}$. Then 
  	\begin{align*}
  	I_{1,n}&=\int_{\Omega}g(x)|u_{n}-u|^{q(x)}dx\\
  	&=\int_{\Omega_{k}}g(x)|u_{n}-u|^{q(x)}dx+\int_{\Omega\backslash\Omega_{k}}g(x)|u_{n}-u|^{q(x)}dx.
  	\end{align*}
  	By Hölder inequality, we get
  	$$
  	\int_{\Omega\backslash\Omega_{k}}g(x)|u_{n}-u|^{q(x)}dx \leq 2^{q^{+}+1}\|g\|_{L^{r(x)}(\Omega\backslash\Omega_{k})}\left(\||u_{n}|^{q(x)}\|_{L^{\frac{h(x)}{q(x)}}(\Omega\backslash\Omega_{k})}+\||u|^{q(x)}\|_{L^{\frac{h(x)}{q(x)}}(\Omega\backslash\Omega_{k})}\right).
  	$$
  	Combining $(\ref{de})$ and $(\ref{10})$  in proof of Lemma $\ref{2}$-$(i)$, we get  
  	\begin{equation}\label{12}
  		\int_{\Omega\backslash\Omega_{k}}g(x)|u_{n}-u|^{q(x)}dx \longrightarrow 0 \quad \text{as }n\longrightarrow+\infty.
  	\end{equation}
  	Since $\quad h(x)< p_{s}^{*}(x)\,$  for all $x\in \Omega$, by Theorem $\ref{ba}$, we conclude that the embedding $W^{s,p(x,y)}(\Omega_{k})\hookrightarrow L^{h(x)}(\Omega_{k})$ is compact. Thus $\,\, u_{n}\rightharpoonup u$ in $W^{s,p(x,y)}(\Omega_{k})$ implies that \begin{center}
  	$\|u_{n}-u\|_{L^{h(x)}(\Omega_{k})}\longrightarrow 0$ as $n\longrightarrow+\infty$.
  	\end{center}	Again, by Hölder inequality, we obtain 
  	$$
  	\int_{\Omega_{k}}g(x)|u_{n}-u|^{q(x)}dx\leq 2\|g\|_{L^{r(x)}(\Omega_{k})}\||u_{n}-u|^{q(x)}\|_{L^{\frac{h(x)}{q(x)}}(\Omega_{k})}.
  	$$ 
  	It follows from Proposition $\ref{11}$ that 
  	$$
  	\||u_{n}-u|^{q(x)}\|_{L^{\frac{h(x)}{q(x)}}(\Omega_{k})}=\left(\|u_{n}-u\|_{L^{h(x)}(\Omega_{k})}\right)^{\overline{q}}\longrightarrow 0\quad \text{as }n\longrightarrow +\infty,
  	$$
  	where $\overline{q} \in [q^{-},q^{+}]$. Since $\|g\|_{L^{r(x)}(\Omega_{k})}$ is bounded, we find that 
  	\begin{equation}\label{13}
  	\int_{\Omega_{k}}g(x)|u_{n}-u|^{q(x)}dx\longrightarrow 0\quad \text{as }n\longrightarrow +\infty.
  	\end{equation}
  	Combining $(\ref{12})$ and $(\ref{13})$, it follows that
  	\begin{equation}\label{14}
  		I_{1,n}\longrightarrow 0 \quad \text{as }n\longrightarrow +\infty.
  	\end{equation}
  By the same miner as above, we get 
  	\begin{equation}\label{15}
  		I_{2,n}\longrightarrow 0 \quad \text{as }n\longrightarrow +\infty.
  	\end{equation}
  	Hence, relations $(\ref{14})$ and $(\ref{15})$ yield
  	$$
  	\lim_{n \to +\infty}\int_{\Omega} f(x,u_{n})(u_{n}-u)\,dx=0.
  	$$
  	Consequently
  	\begin{equation}\label{16}
  		\left\langle \widetilde{\mathscr{L}}(u_{n}),u_{n}-u\right\rangle \longrightarrow 0\quad \text{as }n\longrightarrow +\infty.
  	\end{equation}
   Next, since $u_{n}\rightharpoonup u$, from $(\ref{9})$,  we have 
  	$$
  	\left\langle \mathcal{J}^{'}(u),u_{n}-u\right\rangle \longrightarrow 0\quad \text{as }n\longrightarrow +\infty.
  $$
  Using the same argument as before, we deduce that 
  	\begin{equation}\label{17}
  	\left\langle \widetilde{\mathscr{L}}(u),u_{n}-u\right\rangle \longrightarrow 0\quad \text{as }n\longrightarrow +\infty.
  	\end{equation}
  Using $(\ref{16})$, $(\ref{17})$ and Lemma $\ref{19}$-$(ii)$, as $u_{n}\rightharpoonup u $ in $X_{0}$, we get
  $$
 \left \{ \begin{array} {cl}
  	u_{n}\rightharpoonup u \, \text{ in }X_{0},\\
  	\displaystyle\limsup_{n\to +\infty}\left\langle \widetilde{\mathscr{L}}(u_{n})-\widetilde{\mathscr{L}}(u),u_{n}-u\right\rangle \leq 0,\\
  	\widetilde{\mathscr{L}}\text{ is a mapping of type } (S_{+}).
  \end{array} \right.\Rightarrow  u_{n}\longrightarrow u\,\text{ in }X_{0}. 
  $$
 Therefore, $\mathcal{J}$ satisfies the Palais–Smale condition on $X_0$. By combining this fact with Lemmas \ref{l43} and \ref{l44}, and in light of the mountain pass theorem, we conclude that $u$ is a nontrivial critical point of $\mathcal{J}$. Consequently, $u$ is a nontrivial weak solution to problem $(\mathcal{P}_s)$.
 The proof of Theorem~\ref{3} is complete.
\end{proof}
  
\end{document}